\documentclass[12pt, letterpaper]{article}

\usepackage[left=1in,right=1in, top=1.2in,bottom=1.2in]{geometry}
\usepackage{setspace} 
\usepackage{times} 
\usepackage{hyperref} 

\usepackage{amssymb,amsthm,latexsym,amsmath,epsfig,pgf}
\newtheorem{observation}{Observation}[section]
\newtheorem{theorem}{Theorem}[section]
\newtheorem{lemma}{Lemma}[section]

\newtheorem{problem}{Problem}

\theoremstyle{remark}

\begin{document}

\begin{center}
\Large{\textbf{$D$-Antimagic Labelings of Oriented Star Forests}}\\[1em] 
    \normalsize{
        Ahmad Muchlas Abrar$^{1}$, Rinovia Simanjuntak$^{2,*}$\\[0.5em]
        $^1$Doctoral Program in Mathematics, \\Faculty of Mathematics and Natural Sciences, Institut Teknologi Bandung, \\Bandung, 40132, West Java, Indonesia \\[1em]
        $^2$Combinatorial Mathematics Research Group,\\ Faculty of Mathematics and Natural Sciences, Institut Teknologi Bandung,\\ Bandung, 40132, West Java, Indonesia\\[1em]
        \textit{*Corresponding author:rino@itb.ac.id}
    }
\end{center}

\section*{Abstract}
For a distance set $D$, an oriented graph $\overrightarrow{G}$ is $D$-antimagic if there exists a bijective vertex labeling such that the sum of all labels of $D$-out-neighbors is distinct for each vertex. This paper provides all orientations and all possible $D$s of a $D$-antimagic oriented star. We provide necessary and sufficient condition for $D$-antimagic oriented star forest containing isomorphic oriented stars. We show that for all possible $D$s, there exists an orientation for a star forest to admit a $D$-antimagic labeling.
\\[1em]
\noindent\textbf{Keyword:} $D$-antimagic labeling, oriented graph, oriented star, oriented star forest

\section{Introduction}

Graph labeling is one of the current highly researched topics in graph theory. 
One labeling that has gained significant attention is distance antimagic labeling, introduced by Kamatchi and Arumugam in 2013~\cite{Kamatchi-64-13}. For a simple, undirected graph \( G = (V, E) \), a bijection \( h: V(G) \to \{1, 2, \dots, n\} \) is called a \emph{distance antimagic labeling} if, for each vertex \( u \), the \emph{vertex weight} \( \omega(u) = \sum_{v \in N(u)} h(v) \) is distinct across all vertices; where $N(u)$ is the neighborhood of vertex $u$. 

Distance antimagic labeling has been studied extensively across various graph classes. Kamatchi and Arumugam demonstrated that paths \( P_n \), cycles \( C_n \) (for \( n \neq 4 \)), and wheels \( W_n \) (for \( n \neq 4 \)) are distance antimagic~\cite{Kamatchi-64-13}. Hypercubes \( Q_n \) (\( n \geq 4 \)) were later shown to admit distance antimagic labelings~\cite{Kamatchi cube}. Simanjuntak \textit{et al.} further extended the study to circulant graphs~\cite{sy2014distance} and various graph products, including cartesian, strong, direct, lexicographic, corona, and join products~\cite{simanjuntakprod, Handa}. 

Building on this foundation, Simanjuntak \emph{et al.} introduced the \( D \)-antimagic labeling as a natural generalization~\cite{AAC}. Instead of considering only neighbors, a nonempty distance set $D\subseteq \{0,1,2,\dots,\partial\}$, where $\partial=\max \{d(u,v)<\infty|u,v\in V(\overrightarrow{G})\}$, defines the \emph{\( D \)-neighborhood} of a vertex \( u \), which includes all vertices \( v \) such that \( d(v, u) \in D \). A bijection \( f: V(G) \to \{1, 2, \dots, n\} \) is a \textit{\( D \)-antimagic labeling} if the \textit{\( D \)-weight} \( \omega_D(u) = \sum_{v \in N_D(u)} f(v) \) is distinct for every \( u \in V(G) \). It is clear that a distance antimagic labeling is a special case of a \( D \)-antimagic labeling where \( D = \{1\} \).  

We adapted the aforementioned labeling notion to oriented graphs in ~\cite{Abrarlinearforest} by defining $D$-antimagic labeling for an oriented graph \( \overrightarrow{G}\). 
The $D$-out-neighborhood (or simply $D$-neighborhood) of a vertex $u$ in \( \overrightarrow{G}\) is defined as $N_D(u)=\{v\in V(\overrightarrow{G})\mid d(u,v)\in D\}$. A bijective vertex labeling $g$ such that the $D$-weight differs for each vertex of $\overrightarrow{G}$, is called a $D$-antimagic labeling of $\overrightarrow{G}$; and such a $\overrightarrow{G}$ is called a \textit{$D$-antimagic graph}. 

From ~\cite{Abrarlinearforest}, we have the following trivial, yet useful, lemma.
\begin{lemma}\label{lem:all are 0}
All oriented graphs are $\{0\}$-antimagic.
\end{lemma}

This paper focuses on the labeling properties of oriented star forests, which are collections of disjoint union of oriented stars. In Section \ref{sec:star}, we characterize all orientations and $D$s, such that an oriented star is $D$-antimagic. In Section \ref{sec:starforest}, provide necessary and sufficient condition for $D$-antimagic oriented star forest containing isomorphic oriented stars. We also show that for all possible $D$s, there exists an orientation for a star forest to admit a $D$-antimagic labeling. Throughout this paper, in the figures visualizing a $D$-antimagic labeling of a graph, the labels are represented by black numbers and the other colored numbers in brackets are the $D$-weights of each vertex.

\section{$D$-Antimagic Oriented Stars} \label{sec:star}

Since the diameter of an undirected star is 2, we only deal with the distance set $D$s where $max(D)\leq 2$ when we study $D$-antimagic oriented stars. 

For $n \geq 1$, in any orientation of a star $K_{1,n}$, we call the vertex with degree $n$ the \textit{center}, denoted as $v$. Each vertex with degree one is called a \textit{leaf}, denoted as $v_i$, for $1 \leq i \leq n$, where the index is arranged so that the first $t$ vertices are sources (with out-degree 1) and the remaining $n-t$ vertices are sinks (with in-degree 1). In this case, the $\{1\}$- and $\{2\}$-neighborhoods for each vertex are as follows:
\begin{itemize}
    \item \textbf{For the center:} $N_{\{1\}}(v) = \{v_i \mid t+1 \leq i \leq n\}$ and $N_{\{2\}}(v) = \emptyset$.
    \item \textbf{For the leaves that are sources:} $N_{\{1\}}(v_i) = \{v\}$ and $N_{\{2\}}(v_i) = \{v_j|t+1\le j\le n\}$.
    \item \textbf{For the leaves that are sinks:}  $N_{\{1\}}(v_i)=N_{\{2\}}(v_i) =\emptyset$.
\end{itemize}

In this section, we address all possible distance sets, that is, $$D\in \{\{0\},\{1\},\{2\},\{0,1\},\{0,2\},\{1,2\},\{0,1,2\}\}.$$ Since, by Lemma \ref{lem:all are 0}, any oriented star is $\{0\}$-antimagic, we start with $max(D)=1$.

\subsection{$D$-Antimagic Oriented Stars with $max(D)=1$}

Here we characterize $D$-antimagic oriented stars, where $max(D)=1$, that is, $D \in \{\{1\}, \{0,1\}\}$.

\begin{theorem}
The oriented star $\overrightarrow{K_{1,n}}$ is $\{1\}$-antimagic if and only if
\begin{enumerate}
    \item $n = 1$, or
    \item $n = 2$ and $t = 1$.
\end{enumerate}
\end{theorem}
\begin{proof}
The sufficiency can be proved by constructing $\{1\}$-antimagic labelings for both cases, as provided in Figure \ref{fig:D=1}.
\begin{figure}[ht]
    \centering
    \includegraphics[width=0.6\linewidth]{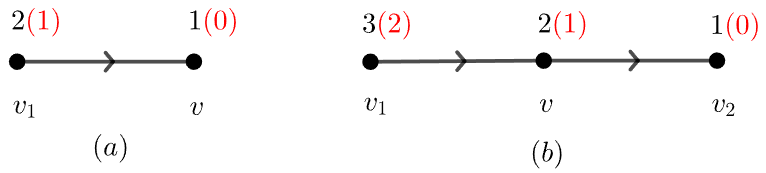}
    \caption{$\{1\}$-antimagic labelings of $\overrightarrow{K_{1,n}}$ where (a) $n = 1$ and (b) $n = 2$ and $t = 1$.}
    \label{fig:D=1}
\end{figure}

    For the necessity, let $\overrightarrow{K_{1,n}}$ be $\{1\}$-antimagic graph. Assume to the contrary the following cases:
    \begin{enumerate}
        \item $n = 2$ and $t = 0$, or
        \item $n = 2$ and $t = 2$, or
        \item $n \geq 3$.
    \end{enumerate}
    
    \textbf{Case 1.} If $n = 2$ and $t = 0$, then the center is a source and the two leaves $v_1, v_2$ are sinks. For any bijection, $\omega_{\{1\}}(v_1) = \omega_{\{1\}}(v_2) = 0$, a contradiction. 
    
    \textbf{Case 2.} If $n = 2$ and $t = 2$, then the center $v$ is a sink and the two leaves $v_1, v_2$ are sources. For any bijection $h$, $\omega_{\{1\}}(v_1) = \omega_{\{1\}}(v_2) = h(v)$, another contradiction. 
    
    \textbf{Case 3.} If $n \geq 3$, by the pigeonhole principle, there must be at least two leaves that are either sinks or sources. If the two leaves are sinks, both weights are zero in any bijection. If the two leaves are sources, their $\{1\}$-weights are the label of the center. These lead to the last contradiction.
\end{proof}

\begin{theorem}
    In any orientation, an oriented star is $\{0,1\}$-antimagic.
\end{theorem}
\begin{proof}
Define a bijection $\alpha: V(\overrightarrow{K_{1,n}}) \rightarrow \{1, 2, \dots, n, n+1\}$ such that $\alpha(v_i) = i$, for $1 \leq i \leq n$, and $\alpha(v) = n+1$. 
We consider three cases based on the value of $t$,

\textbf{Case 1: $t = 0$.} The $\{0,1\}$-neighborhoods of all vertices are:  
$N_{\{0,1\}}(v) = V(\overrightarrow{K_{1,n}})$ and $N_{\{0,1\}}(v_i) = \{v_i\}$, for $1 \leq i \leq n$.    Thus, $\omega_{\{0,1\}}(v) = 
\frac{(n+2)(n+1)}{2}$ and $\omega_{\{0,1\}}(v_i) = \alpha(v_i) = i$, for $i = 1, \dots, n$. Since $\frac{(n+2)(n+1)}{2} > n$, each vertex has distinct $\{0,1\}$-weight.

\textbf{Case 2: $2 \leq t \leq n-1$.} The $\{0,1\}$-neighborhoods and $\{0,1\}$-weights of each vertex are:
\begin{description}
\item[For the center:] $N_{\{0,1\}}(v)= \{v\} \cup \{v_i \mid t+1 \leq i \leq n\}$ and $\omega_{\{0,1\}}(v) = (n+1) + \frac{(n-t)(n+t+1)}{2}$.
\item[For the leaves that are sources:] $N_{\{0,1\}}(v_i) = \{v_i, v\}$ and $\omega_{\{0,1\}}(v_i) = (n+1) + i$, for $1 \leq i \leq t$,
\item[For leaves that are sinks:] $ N_{\{0,1\}}(v_i) = \{v_i\}$ and $\omega_{\{0,1\}}(v_i) = i$, $t+1 \leq i \leq n$.
\end{description}

\textbf{Case 3: $t = n$.} Here we have $N_{\{0,1\}}(v) = \{v\}$ and $N_{\{0,1\}}(v_i) = \{v, v_i\}$, for $1 \leq i \leq n$. Thus, $\omega_{\{0,1\}}(v) = n+1$ and $\omega_{\{0,1\}}(v_i) = n + i + 1$, for $1 \leq i \leq n$. 

All the $\{0,1\}$-weights are distinct according to the following order:
$$\omega_{\{0,1\}}(v_{t+1}) < \omega_{\{0,1\}}(v_{t+2}) < \dots < \omega_{\{0,1\}}(v_{n}) < \omega_{\{0,1\}}(v_1) < \omega_{\{0,1\}}(v_2) <\dots < \omega_{\{0,1\}}(v_{t}).$$

Figure \ref{fig:D=0,1 K_{1,5}} depicts an example of this labeling. 

\end{proof}

\begin{figure}[ht]
      \centering      \includegraphics[width=0.35\linewidth]{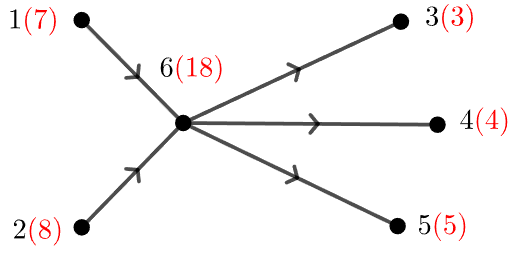}
      \caption{A $\{0,1\}$-antimagic labeling of a $\overrightarrow{K_{1,5}}$.}
      \label{fig:D=0,1 K_{1,5}}  
\end{figure}

\subsection{$D$-Antimagic Oriented Stars with $max(D)=2$}

This subsection characterizes $D$-antimagic oriented stars when $max(D)=2$, that is, 
$$D \in \{\{2\},\{0,2\},\{1,2\},\{0,1,2\}\}.$$ 

Note that distance two is achievable only between two leaves. Consequently, 

\begin{observation}\label{ob:center}
An oriented star admits distance two if and only if its center is neither a source nor a sink.    
\end{observation}

In this case, the number of leaves that are sources is $1\le t\le n-1$ and the $\{2\}$-neighborhood of each vertex in $\overrightarrow{K_{1,n}}$ is given by: $N_{\{2\}}(v) = N_{\{2\}}(v_i) = \emptyset$, for $t+1 \leq i \leq n$, and 
$N_{\{2\}}(v_i) = \{v_j \mid t+1 \leq j \leq n\}$, for $1 \leq i \leq t$. 
Clearly, the center and the leaves that are sinks have the same $\{2\}$-weights that leads to the following:

\begin{lemma}
In any orientation, an oriented star is not a $\{2\}$-antimagic.
\end{lemma}

Although any oriented star is not $\{2\}$-antimagic, positive results exist for the other three distance sets. We characterize all in the following three theorems.

\begin{theorem}
For $n \geq 2$, an oriented star $\overrightarrow{K_{1,n}}$ is $\{0,2\}$-antimagic if and only if its center is neither a source nor a sink.
\end{theorem}
\begin{proof}  
By Observation \ref{ob:center}, the proof is complete by constructing a $\{0,2\}$-antimagic labeling for an oriented star center that is neither a source nor a sink.

Define a bijection $\beta: V(\overrightarrow{K_{1,n}}) \to \{1, 2, \dots, n, n+1\}$ by $\beta(v) = t+1$, $\beta(v_i) = i$, for $1 \leq i \leq t$, and $\beta(v_i) = i+1$ for $t+1 \leq i \leq n$. 
Thus, the $\{0,2\}$-neighborhoods and $\{0,2\}$-weights are as follows:
\begin{description}
\item[For the center:] $N_{\{0,2\}}(v) = \{v\}$ and $\omega_{\{0,2\}}(v) = t+1$.
\item[For the sources:] $N_{\{0,2\}}(v_i) = \{v_i\} \cup \{v_j \mid t+1 \leq j \leq n\}$ and $ \omega_{\{0,2\}}(v_i) =  i + \frac{(n-t)(n+t+1)}{2}$, for $1 \leq i \leq t$.
\item[For the sinks:] $N_{\{0,2\}}(v_i) = \{v_i\}$ and $\omega_{\{0,2\}}(v_i) = i+1$, $t+1 \leq i \leq n$.
\end{description}

All the $\{0,2\}$-weights are distinct since $\omega_{\{0,2\}}(v_i)>\omega_{\{0,2\}}(v_{i'})>\omega_{\{0,2\}}(v)$ for $1\le i'<i\le n$. 

See Figure \ref{fig:D=0,2 K_{1,5}} for an example.
\end{proof}

\begin{figure}[ht]
    \centering    \includegraphics[width=0.4\linewidth]{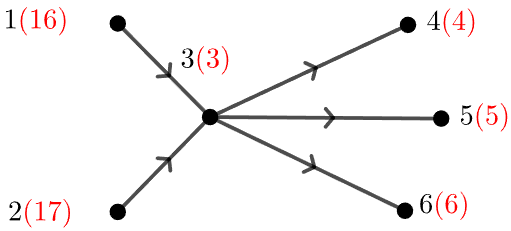}
    \caption{A $\{0,2\}$-antimagic labeling of a $\overrightarrow{K_{1,5}}$.}
    \label{fig:D=0,2 K_{1,5}}
\end{figure}


\begin{theorem}
For $n \geq 2$, an oriented star $\overrightarrow{K_{1,n}}$ is $\{1,2\}$-antimagic if and only if $n = 2$ and its center is neither a source nor a sink.
\end{theorem}
\begin{proof}
The $\{1,2\}$-neighborhoods of each vertex are: $N_{\{1,2\}}(v) = \{v_j \mid t+1 \leq j \leq n\}$, $N_{\{1,2\}}(v_i) = \{v\} \cup \{v_j \mid t+1 \leq j \leq n\}$, for $1 \leq i \leq t$, and $N_{\{1,2\}}(v_i) = \emptyset$, for $t+1 \leq i \leq n$.
If $\overrightarrow{K_{1,n}}$ is $\{1,2\}$-antimagic, to ensure that there is only one vertex with zero $\{1,2\}$-weight, then there should be only one sink. In this case, to ensure that there is only one vertex with $\{1,2\}$-weight equal to the sum of the labels of the center and the sink, there should also be only one source.
This implies $n = 2$ and Observation \ref{ob:center} completes the sufficiency.

For the construction of a $\{1,2\}$-antimagic labeling of the $\overrightarrow{K_{1,2}}$ whose center is neither a source nor a sink, refer to Figure \ref{fig:D=1,2}.
\end{proof}

\begin{figure}[ht]
    \centering
    \includegraphics[width=0.35\linewidth]{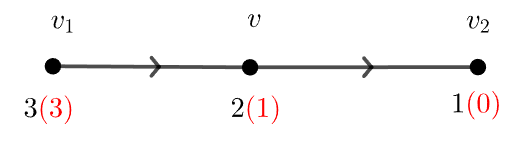}
    \caption{A $\{1,2\}$-antimagic labeling of the $\overrightarrow{K_{1,2}}$.}
    \label{fig:D=1,2}
\end{figure}


\begin{theorem}
For $n \geq 2$, an oriented star $\overrightarrow{K_{1,n}}$ is $\{0,1,2\}$-antimagic if and only if its center is neither a source nor a sink.
\end{theorem}
\begin{proof}
Due to Observation \ref{ob:center}, the proof completes by constructing a $\{0,1,2\}$-antimagic labeling of a $\overrightarrow{K_{1,n}}$ whose center is neither a source nor a sink.
Define a bijection $\beta: V(\overrightarrow{K_{1,n}}) \to \{1, 2, \dots, n, n+1\}$ such that: $\beta(v) = t+1$ and
$$
\beta(v_i)=\begin{cases}
			i, & \text{for } 1 \leq i \leq t,\\
            i+1, & \text{for } t+1 \leq i \leq n.
		 \end{cases}
$$

Thus, the $\{0,1,2\}$-neighborhood and $\{0,1,2\}$-weight of each vertex are:
\begin{description}
\item[For the center:] $N_{\{0,1,2\}}(v) = \{v\} \cup \{v_j \mid t+1 \leq j \leq n\}$ and $\omega_{\{0,1,2\}}(v)=(n-t)\frac{n+t+3}{2}+t+1$.
\item[For the sources:] $N_{\{0,1,2\}}(v_i) = \{v\} \cup \{v_i\} \cup \{v_j \mid t+1 \leq j \leq n\}$ and $\omega_{\{0,1,2\}}(v_i)=i+t+1+(n-t)\frac{n+t+3}{2}$, $1\le i \le t$.
\item[For the sinks:] $N_{\{0,1,2\}}(v_i) = \{v_i\}$ and $\omega_{\{0,1,2\}}(v_i) =i+1$, $t+1 \leq i \leq n$.
\end{description}

It can be verified that the $\{0,1,2\}$-weights of all vertices are distinct, where: $\omega_D(v_i)>\omega_D(v_{i'})$, for $i>i'$ and $\omega_D(v)\ne \omega_D(v_i)$, for $1\le i\le n$. 

Figure \ref{fig:D=0,1,2 K1,5} provides an example of this labeling.
\end{proof}

\begin{figure}[ht]
    \centering    \includegraphics[width=0.35\linewidth]{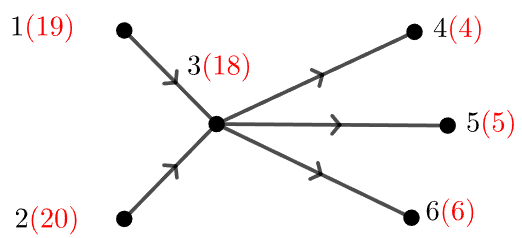}
    \caption{A \{0,1,2\}-antimagic labeling of a $\overrightarrow{K_{1,5}}$.}
    \label{fig:D=0,1,2 K1,5}
\end{figure}

\section{$D$-Antimagic Labelings of Oriented Star Forests} \label{sec:starforest}

We define a \textit{star forest} as a disjoint union of at least two stars. We begin by posing a necessary condition for a $D$-antimagic oriented star forest.

\begin{lemma}\label{lem:minD=0}
If an oriented star forest is $D$-antimagic then $\min(D)=0$.
\end{lemma}
\begin{proof}
     Since every star has at least one sink, if $\min(D)\ge 1$, any star forest will have at least two sinks of $D$-weight zero.
\end{proof}

In the next three theorems, we show that the necessary condition in Lemma \ref{lem:minD=0}, together with the condition in Observation \ref{ob:center}, is also sufficient for oriented star forests consisting of isomorphic oriented stars. For $m\ge 2$ and $n\ge 1$, let $m\overrightarrow{K_{1,n}}$ be an oriented star forest consisting of isomorphic oriented stars with a vertex set $\{v^j,v_i^j|1\le j\le m, 1\le i\le n\}$, where $v^j$ is the center of each star component.

\begin{theorem}\label{th:mkn 0,1}
For $m\ge 2$ and $n\ge 1$, in any orientation, $m\overrightarrow{K_{1,n}}$ is $\{0\}$- and $\{0,1\}$-antimagic.
\end{theorem}
\begin{proof} By Lemma \ref{lem:all are 0}, $m\overrightarrow{K_{1,n}}$ is $\{0\}$-antimagic. For $D=\{0,1\}$, consider the following three cases based on the number of source leaves $t$. 

\noindent \textbf{Case 1: $t=0$.}
For all $i$ and $j$, define a vertex labeling $h_1$ where $h_1(v^j)=mn+j$ and $h_1(v_i^j)=n(j-1)+i$. Thus, for $1\le j\le m$, $D$-neighborhood and $D$-weight of each vertex are:
\begin{description}
\item[For the centers:] $N_D(v^j)=\{v_i^j,v^j|1\le i\le n\}$ and $\omega_D(v^j)=(n^2+1)j+\frac{n(2m-n+1)}{2}$.
\item[For the leaves:] $N_D(v_i^j)=\{v_i^j\}$ and $\omega_D(v_i^j)=n(j-1)+i$, for $1\le i\le n$.
\end{description}

All the $D$-weights are different since for $1\le j,j'\le m$ the following holds:
\begin{itemize}
\item $\omega_D(v^j)>\omega_D(v^{j'})$, when $j>j'$.
\item \textbf{For $1\le i,i'\le n$:} $\omega_D(v_i^j)<\omega_D(v^1)$; $\omega_D(v_{i'}^{j'})>\omega_D(v_{i}^{j})$, when $j'>j$ or when $j'=j$ and $i'>i$. 

\end{itemize}

\noindent \textbf{Case 2: $2\le t\le n-1$.}
For all $1\le j\le m$, define a vertex labeling $h_2$ by: $h_2(v^j)=mn+j$ and
$$
h_2(v_i^j)=\begin{cases}
			i+j(t-1), & \text{for } 1 \leq i \leq t,\\
            m(i-1)+j, & \text{for } t+1 \leq i \leq n.
		 \end{cases}
$$
Thus, for $1\le j\le m$, we obtain the following $D$-neighborhoods and $D$-weights:
    \begin{description}
    \item[For the centers:]  $N_D(v^j)=\{v_p^j,v^j| t+1\le p\le n\}$ and $\omega_D(v^j)=j(n-t+1)+\frac{m(n^2-t^2+n+t)}{2}$.
    \item[For the source leaves:]
    $N_D(v_i^j)=\{v_i^j,v^j\}$ and $\omega_D(v_i^j)=mn+jt+i$, for $1\le i\le t$.
    \item[For the sink leaves:]
    $N_D(v_i^j)=\{v_i^j\}$ and $\omega_D(v_i^j)=m(i-1)+j$, for $t+1\le i\le n$.
    \end{description}
Here each $D$-weight is unique, since for $1\le j,j'\le m$, we have:
    \begin{itemize}
    \item $\omega_D(v^j)>\omega_D(v^{j'})$, when $j>j'$.
    \item \textbf{For $1\le i,i'\le t$:} $\omega_D(v_{i}^j)<\omega_D(v^1)$;  $\omega_D(v_{i'}^{j'})>\omega_D(v_{i}^{j})$, when $j'>j$ or when $j'=j$ and $i'>i$. 
    \item \textbf{For $t+1\le i,i'\le n$:} $\omega_D(v_{i}^j)<\omega_D(v^1)$; $\omega_D(v_{i'}^{j'})>\omega_D(v_{i}^{j})$, when $j'>j$ or when $j'=j$ and $i'>i$. 
    \end{itemize}
    
\noindent \textbf{Case 3: $t=n$.} For all $i$ and $j$, define a vertex labeling $h_3$ by  $h_3(v^j)=j$ and $h_3(v_i^j)=mi+j$. Thus, for $1\le j\le m$, we obtain the distinct $D$-weights as follows:
\begin{description}
\item[For the centers:]  $N_D(v^j)=\{v^j\}$ and $\omega_D(v^j)=j$.
\item[For the source leaves:] $N_D(v_i^j)=\{v_i^j,v^j\}$ and $\omega_D(v_i^j)=mi+2j$, where $\omega_D(v_i^j)<\omega_D(v_{i'}^{j'})$ when $i<i'$ or when $j<j'$ and $i=i'$, for $1\le i\le n$.
\end{description} 

The three cases complete the proof and refer to Figure \ref{fig:D0,1 mkn} for an example.
\end{proof}

    \begin{figure}[ht]
    \vspace{-0.5em}
        \centering        \includegraphics[width=1\linewidth]{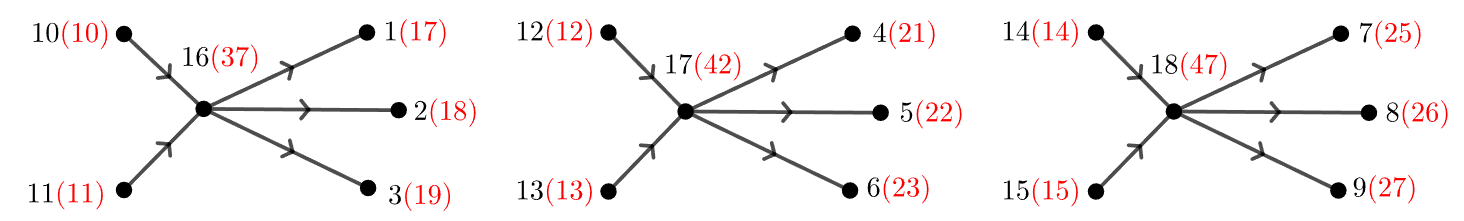}
        \caption{A $\{0,1\}$-antimagic labeling of a $3\overrightarrow{K_{1,5}}$.}
        \label{fig:D0,1 mkn}
    \end{figure}
    
\begin{theorem}\label{th:mkn 0,2 dan 0,1,2}
For $m\ge 2$ and $n\ge 1$, $m\overrightarrow{K_{1,n}}$ is $\{0,2\}$- and $\{0,1,2\}$-antimagic if and only if each center is neither a sink nor a source.
\end{theorem}
\begin{proof} For $1\le j\le m$, define a bijection $g_2$ by: $g_2(v^j)=m(n-t)+j$ and
$$
g_2(v_i^j)=\begin{cases}
			m(n-t+1)+t(j-1)+i, & \text{for } 1 \leq i \leq t,\\
            m(i-t-1)+j, & \text{for } t+1 \leq i \leq n.
		 \end{cases}
$$
Since $2\in D$ then $1\leq t\leq n-1$. For $1\le j\le m$, we shall determine the $D$-neighborhoods and $D$-weights of each vertex. For $D=\{0,2\}$, we obtain:
\begin{itemize}
\item \textbf{For the centers:} $N_D(v^j)=\{v^j\}$ and $\omega_D(v^j)=m(n-t)+j$.
\item \textbf{For the source leaves:} $N_D(v_i^j)=\{v_p^j|t+1\le p\le n\text{ or } p=i\}$ and $\omega_D(v_i^j)=nj+i-t+\frac{m(n-t)(n-t+1)}{2}$, for $1\le i\le t$.
\item \textbf{For the sink leaves:} $N_D(v_i^j)=\{v_i^j\}$ and $\omega_D(v_i^j)=m(i-t-1)+j$, for $t+1\le i\le n$.
\end{itemize}
While for $D=\{0,1,2\}$:
\begin{itemize}
\item \textbf{For the centers:} $N_D(v^j)=\{v^j,v_p^j|t+1\le p\le n\}$ and $\omega_D(v^j)=(n-t)\left(j+\frac{m(n-t+1)}{2}\right)+j$.
\item \textbf{For the source leaves:} $N_D(v_i^j)=\{v^j\}\cup\{v_p^j|t+1\le p\le n\text{ or } p=i\}$ and $\omega_D(v_i^j)=i+j+m++t(j-1)+(n-t)\left(j+\frac{m}{2}(n-t+3)\right)$, for $1\le i\le t$.
\item \textbf{For the sink leaves:} $N_D(v_i^j)=\{v_i^j\}$ and $\omega_D(v_i^j)=m(i-t-1)+j$, for $t+1\le i\le n$.
\end{itemize}
To complete the proof, for $1\le j\le m$, we order all $D$-weights for both $D$s:
\begin{itemize}
\item $\omega_D(v^j)<\omega_D(v^{j'})$, when $j<j'$.
\item \textbf{For $1\le i\le t$:} $\omega_D(v_i^j)>\omega_D(v^m)$; $\omega_D(v_i^j)<\omega_D(v_{i'}^{j'})$, when $j<j'$ or when $j'=j$ and $i<i'$.
\item \textbf{For $t+1\le i\le n$:}  $\omega_D(v_i^j)<\omega_D(v^1)$; $\omega_D(v_i^j)<\omega_D(v_{i'}^{j'})$, when $i<i'$ or when $i'=i$ and $j<j'$. 
\end{itemize}

An example of these labelings is provided in Figure \ref{fig:D=0,2 and 0,1,2 of mK1,5}. \end{proof}

\begin{figure}[ht]
    \centering    \includegraphics[width=1\linewidth]{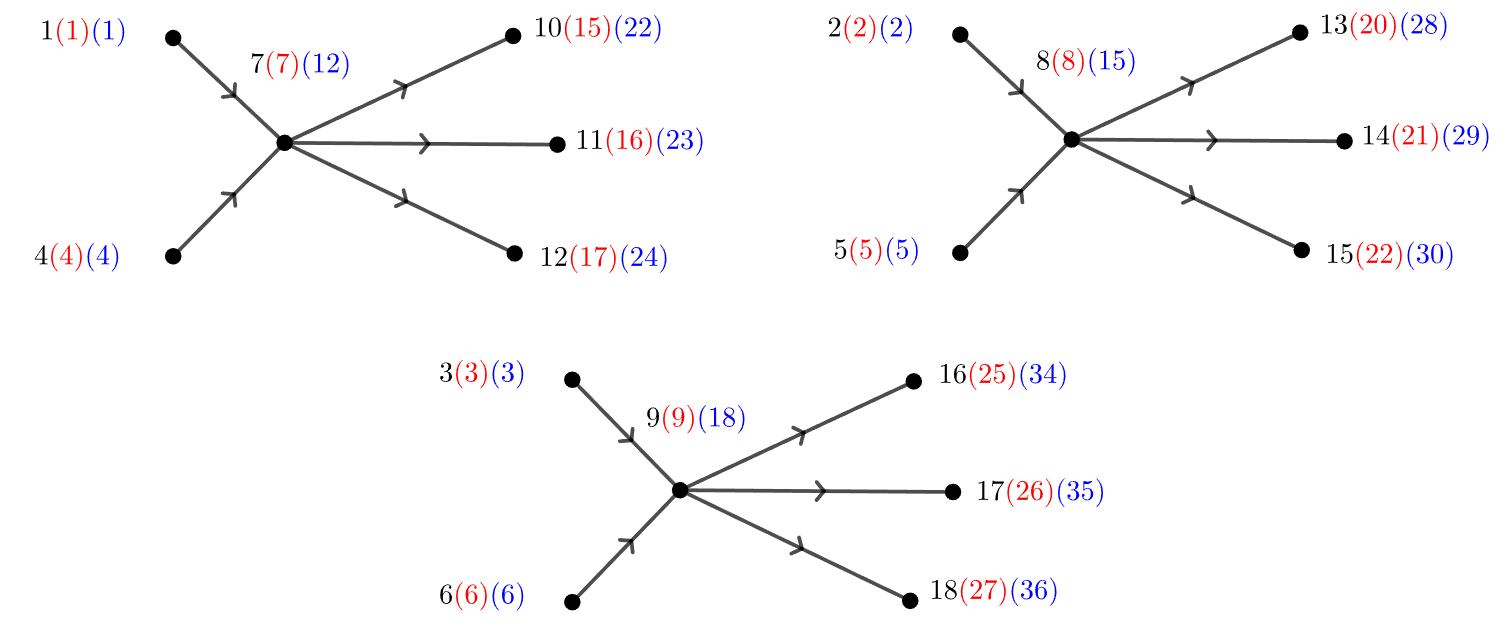}
    \caption{$\{0,2\}$- and $\{0,1,2\}$-antimagic labelings of a $3\overrightarrow{K_{1,5}}$. (The black numbers are the labels, the \textcolor{red}{red numbers} are the $\{0,2\}$-weights and the \textcolor{blue}{blue numbers} are the $\{0,1,2\}$-weights of each vertex.)}
    \label{fig:D=0,2 and 0,1,2 of mK1,5}
\end{figure}

We conclude by exploring $D$-antimagic labelings of a star forest consisting of at least two (not necessarily isomorphic) oriented stars. For $\tau \geq 1$, $1 \leq j \leq \tau$, $1\leq n_1 < n_2 < \dots < n_{j-1} < n_j$, $m_j \geq 1$, and $\sum_{j=1}^{\tau}m_j\ge 2$, let $S=\bigcup_{j=1}^{\tau}\overrightarrow{m_jK_{1,n_j}}$ be an oriented star forest with vertex set $\{v_{i,s}^j,v_s^j|1\le j\le \tau, 1\le s\le m_j, 1\le i\le n_j\}$, where $v_s^j$ is the center of each star. 

\begin{theorem}
There exists an orientation of the star forest $S$ such that $S$ is $\{0,1\}$-, $\{0,2\}$-, and $\{0,1,2\}$-antimagic.
\end{theorem}
\begin{proof} We define a $\Pi$ orientation of $S$ as the orientation with the arc set $\{(v_{i,s}^j,v_s^j),(v_s^j,v_{n_j,s}^j)\mid 1\le j\le \tau, 1\le i\le n_j-1\}$. For $1\le j\le \tau$, $1\le s\le m_j$, and $m_0=n_0=0$, we also define a bijective vertex labeling $h$ by:
\begin{align*}
h(v_s^j)=&\sum_{p=1}^{\tau}m_p+\sum_{p=0}^{j-1}m_p+s,\\
h(v_{i,s}^j)=&2\sum_{p=1}^{\tau}m_j+\sum_{p=0}^{j-1}m_p(n_p-1)+(s-1)(n_j-1)+i, \text{ for } 1\le i\le n_j-1, \text{ and}\\
h(v_{n_j,s}^j)=&\sum_{p=0}^{j-1}m_p+s.
\end{align*}

For $1\le j\le \tau$ and $1\le s\le m_j$, we shall determine $D$-neighborhoods and $D$-weights for each vertex separately for each distance set $D$.

\noindent \textbf{Case 1: $D=\{0,1\}$.}
\begin{itemize}
    \item \textbf{For $1\le i\le n_j-1$:} $N_D(v_{i,s}^j)=\{v_{i,s}^j,v_s^j\}$ and $$\omega_D(v_{i,s}^j)=3\sum_{p=1}^{\tau}m_j+\sum_{p=0}^{j-1}m_p n_p+(s-1)n_j+i+1.$$
    \item $N_D(v_s^j)=\{v_s^j,v_{n_j,s}^j\}$ and $\omega_D(v_s^j)=\sum_{p=1}^{\tau}m_p+2\left(\sum_{p=0}^{j-1}m_p+s\right)$.
    \item $N_D(v_{n_j,s}^j)=\{v_{n_j,s}^j\}$ and $\omega_D(v_{n_j,s}^j)=\sum_{p=0}^{j-1}m_p+s$.
\end{itemize}

\noindent \textbf{Case 2: $D=\{0,2\}$.}
    \begin{itemize}
    \item \textbf{For $1\le i\le n_j-1$:} $N_D(v_{i,s}^j)=\{v_{i,s}^j,v_{n_j,s}^j\}$ and $$\omega_D(v_{i,s}^j)=2\sum_{p=1}^{\tau}m_j+\sum_{p=0}^{j-1}m_p n_p+(s-1)n_j+i+1.$$
    \item $N_D(v_s^j)=\{v_{s}^j\}$ and $\omega_D(v_s^j)=\sum_{p=1}^{\tau}m_p+\sum_{p=0}^{j-1}m_p+s$.
    \item $N_D(v_{n_j,s}^j)=\{v_{n_j,s}^j\}$ and $\omega_D(v_{n_j,s}^j)=\sum_{p=0}^{j-1}m_p+s$.
\end{itemize}

\noindent    \textbf{Case 3: $D=\{0,1,2\}$.}
    \begin{itemize}
    \item \textbf{For $1\le i\le n_j-1$:} $N_D(v_{i,s}^j)=\{v_{i,s}^j,v_s^j,v_{n_j,s}^j\}$ and $$\omega_D(v_{i,s}^j)=3\sum_{p=1}^{\tau}m_j+\sum_{p=0}^{j-1}m_p(n_p+1)+(s-1)(n_j+1)+i+2.$$
    \item $N_D(v_s^j)=\{v_s^j,v_{n_j,s}^j\}$ and $\omega_D(v_s^j)=\sum_{p=1}^{\tau}m_p+2\left(\sum_{p=0}^{j-1}m_p+s\right)$.
    \item $N_D(v_{n_j,s}^j)=\{v_{n_j,s}^j\}$ and $\omega_D(v_{n_j,s}^j)=\sum_{p=0}^{j-1}m_p+s$.
\end{itemize}

For $1\le j\le \tau$, $1\le s\le m_j$, and 
all three possible $D$s, it can be verified that:
\begin{description}
    \item[For the centers:] $\omega_D(v_s^j)<\omega_D(v_{s'}^{j'})$, when $j>j'$ or when $j=j'$ and $s'>s$. 
    Also, $\omega_D(v_{m_t}^t)<\omega_D(v_{1,1}^1)$.
    \item[For $1\le i\le n_j-1$:] $\omega_D(v_{i,s}^j)<\omega_D(v_{i',s'}^{j'})$, when $j>j'$ or when $j=j'$ and $s'>s$ or when $j=j', s=s'$, and $i'>i$. 
    \item[For $i=n_j$:]  $\omega_D(v_{n_j,s}^j)<\omega_D(v_{n_j,s'}^{j'})$, $j>j'$ or when $j=j'$ and $s'>s$. Also, $\omega_D(v_{n_t,m_t}^t)<\omega_D(v_1^1)$.
    \end{description}
    
Therefore, in orientation $\Pi$, $S$ is $\{0,1\}$-,$\{0,2\}$-, and $\{0,1,2\}$-antimagic. Figure \ref{fig:S} provides an example of these labelings.
\end{proof}

\begin{figure}[ht]
    \centering
    \includegraphics[width=1\linewidth]{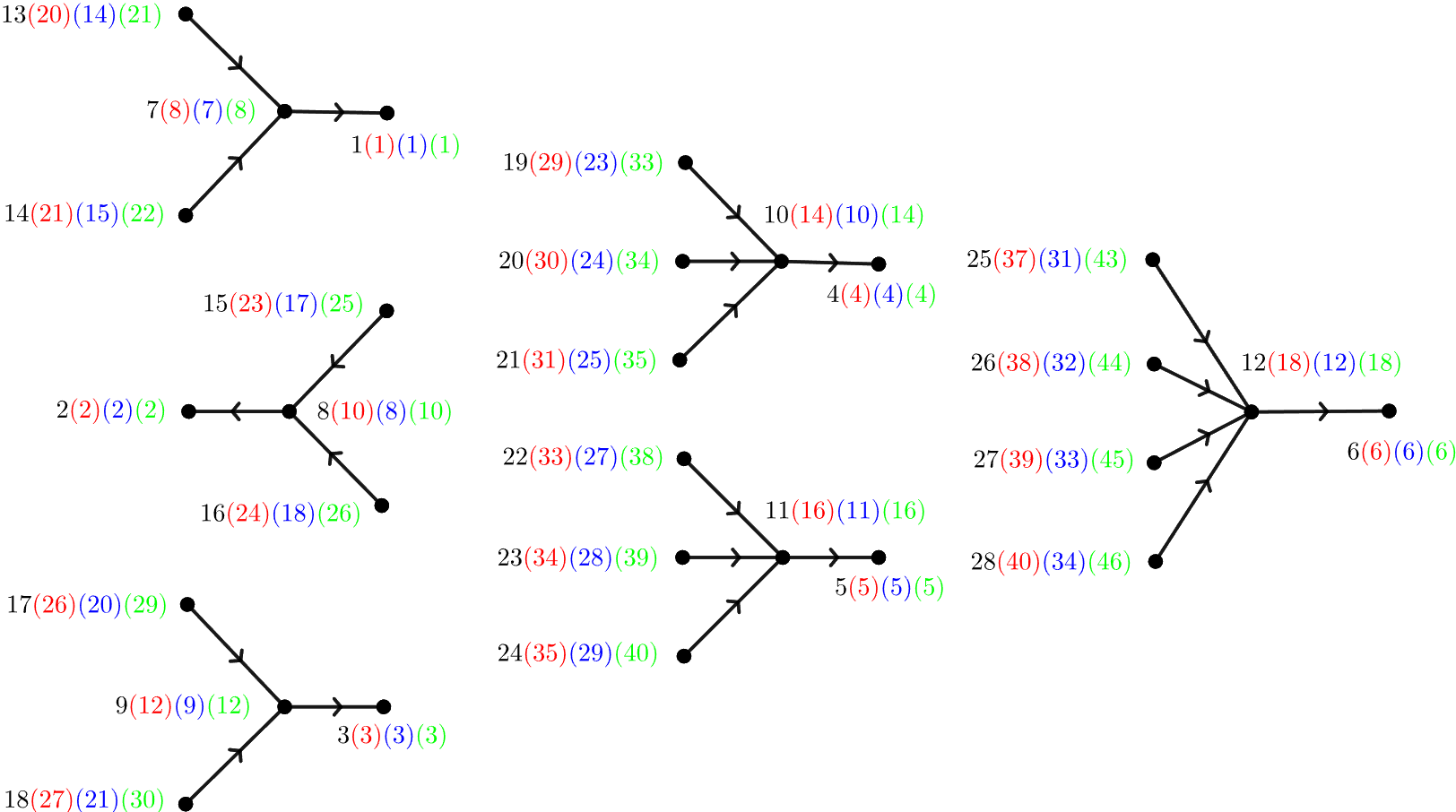}
    \caption{$\{0,1\}$-, $\{0,2\}$-, and $\{0,1,2\}$-antimagic labelings of a $\Pi$-oriented $\overrightarrow{3K_{1,3}}\cup \overrightarrow{2K_{1,4}}\cup \overrightarrow{K_{1,5}}$. (The black numbers are the labels, the \textcolor{red}{red numbers} are the $\{0,1\}$-weights, the \textcolor{blue}{blue numbers} are the $\{0,2\}$-weights, and the \textcolor{green}{green numbers} are the $\{0,1,2\}$-weights of each vertex.)}
    \label{fig:S}
\end{figure}

\begin{figure}[ht]
    \centering    \includegraphics[width=1\linewidth]{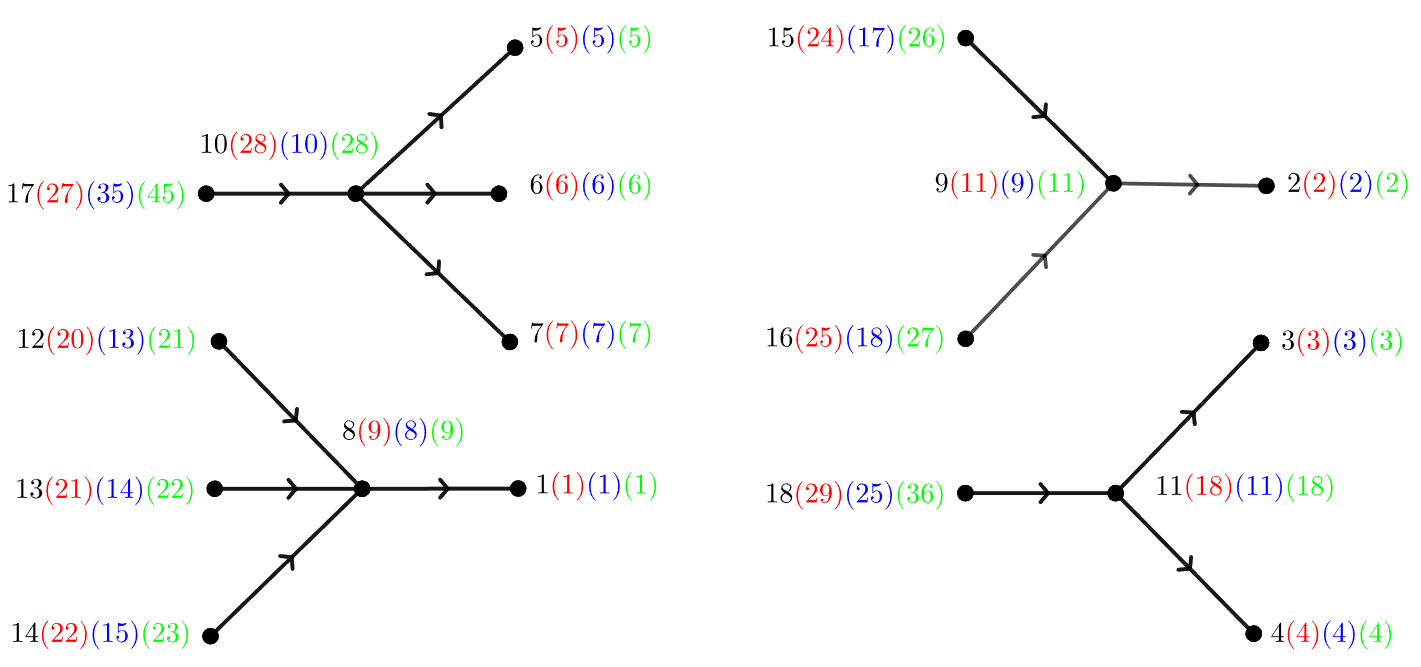}
    \caption{$\{0,1\}$-, $\{0,2\}$-, and $\{0,1,2\}$-antimagic labelings of a $\overrightarrow{2K_{1,3}}\cup \overrightarrow{2K_{1,4}}$. (The black numbers are the labels, the \textcolor{red}{red numbers} are the $\{0,1\}$-weights, the \textcolor{blue}{blue numbers} are the $\{0,2\}$-weights, and the \textcolor{green}{green numbers} are the $\{0,1,2\}$-weights of each vertex.)}
    \label{fig:S bukan Pi}
\end{figure}
In Figure \ref{fig:S bukan Pi} we have an example of an oriented star forest, with orientation other than $\Pi$, that is $D$-antimagic. Thus, we conclude by posing a general open problem for oriented star forests.
\begin{problem}
List all orientations of a star forest
$S$, such that it is $\{0,1\}$-, $\{0,2\}$-, and $\{0,1,2\}$-antimagic.
\end{problem}

\end{document}